\documentclass[11pt]{article}
\usepackage[a4paper,top=28mm,bottom=28mm,inner=30mm,outer=30mm]{geometry}
 
\usepackage[svgnames,x11names]{xcolor}
\usepackage{amsfonts,amsmath,amsthm,boldline,colortbl,comment,enumerate,float,graphicx,import,makecell,mathtools,multirow,pdflscape,subcaption,tikz}
\mathtoolsset{centercolon}
\usetikzlibrary{through,intersections,calc}
\usepackage{hyperref}

\usepackage{cite}

\newcommand{\numberset}{\mathbb}
\newcommand{\N}{\numberset{N}}

\newcommand{\R}{\numberset{R}}

\newcommand{\F}{\numberset{F}}
\newcommand{\Mat}{\operatorname{Mat}}

\newcommand{\mC}{\mathcal{C}}

\newcommand{\Fq}{\F_q}

\newcommand{\rk}{\textnormal{rk}}

\newcommand{\bfn}{\mathbf {n}}
\newcommand{\bfm}{\mathbf {m}}
\newcommand{\spac}{\Mat(\bfn,\bfm,\Fq)}
\newcommand{\srk}{\mathrm{srk}}

\newcommand{\st}{\,:\,}

\newcommand{\floor}[1]{{\left\lfloor{#1}\right\rfloor}}
\newcommand{\ceil}[1]{{\left\lceil{#1}\right\rceil}}

\newlength{\mynodespace}
\newcommand{\erase}[2][says]

\newtheorem{theorem}{Theorem}
\newtheorem{lemma}[theorem]{Lemma}
\newtheorem{corollary}[theorem]{Corollary}
\newtheorem{proposition}[theorem]{Proposition}

\theoremstyle{definition}
\newtheorem{remark}[theorem]{Remark}
\newtheorem{definition}[theorem]{Definition}
\newtheorem{notation}{Notation}

\title{{Eigenvalue Bounds for Sum-Rank-Metric Codes}}

\author{Aida Abiad\thanks{\texttt{a.abiad.monge@tue.nl},  Department of Mathematics and Computer Science, Eindhoven University of Technology, the Netherlands\\ Department of Mathematics: Analysis, Logic and Discrete Mathematics, Ghent University, Belgium\\ Department of Mathematics and Data Science, Vrije Universiteit Brussel, Belgium}\quad Antonina P. Khramova\thanks{\texttt{a.khramova@tue.nl}, Department of Mathematics and Computer Science, Eindhoven University of Technology, the Netherlands}\quad Alberto Ravagnani\thanks{\texttt{a.ravagnani@tue.nl}, Department of Mathematics and Computer Science, Eindhoven University of Technology, the Netherlands}}

\date{}

\begin{document}

\maketitle

\begin{abstract}
We consider the problem of deriving upper bounds on the parameters of sum-rank-metric codes, with focus on their dimension and block length.
The sum-rank metric is a combination of the Hamming and the rank metric, and most of the available techniques to investigate it seem to be unable to fully capture 
its hybrid nature.
In this paper, we introduce a new approach 
based on sum-rank-metric graphs, in which the vertices are tuples of matrices over a finite field, and where two such tuples are connected when their sum-rank distance is equal to one.
We establish various structural properties of sum-rank-metric graphs and combine them with eigenvalue techniques to obtain bounds on the cardinality of sum-rank-metric codes.
The bounds we derive improve on the best known bounds for several choices of the parameters. While our bounds 
are explicit only for small values of the minimum distance, they clearly indicate that spectral theory is able to capture the nature of the sum-rank-metric better than the currently available methods. They also allow us to establish new non-existence results for (possibly nonlinear) MSRD codes.

\bigskip




\end{abstract}

\bigskip

\section{Introduction}

In the past decade, \textit{sum-rank-metric codes} have featured prominently in the coding theory literature, mainly in connection with multi-shot network coding; see \cite{byrne2021fundamental,MP18,MP20a,MPK19,MPK19a,PRR20,neri2022twisted,martinez2022general,neri2023geometry,bartz2021decoding,ott2022covering,byrne2022anticodes,chen2023new,borello2023geometric,camps2022optimal} among many others.
In that context, sum-rank-metric codes can significantly reduce the alphabet size with respect to schemes based on ordinary rank-metric codes.
Sum-rank-metric codes
have been proposed for other applications as well, including 
space-time coding and 
distributed storage;
see for instance~\cite{ShKsch20,ElGH03,LuKu05,MPK19a}. 

A sum-rank-metric code is a subspace of the Cartesian product of (possibly different) matrix spaces over a finite field $\F_q$. The \textit{sum-rank} of a matrix tuple is the sum of the ranks of its components, and the \textit{sum-rank distance} between matrix tuples is the sum-rank of their difference. A code has good correction capability when it has large cardinality and all distinct matrix tuples are at large 
sum-rank distance from each other.

When investigating the sum-rank metric, it is often helpful to observe that it is a \textit{hybrid} between the 
rank and the Hamming metric; see~\cite{macwilliams1977theory} for a general reference about the latter. In fact, sum-rank-metric codes exhibit a behaviour
similar to Hamming-metric codes when all the the matrix spaces in the Cartesian product are small, and similar to rank-metric codes when the number of matrix spaces in the product is small. We are however not aware of any rigorous formalization of this behavior. Furthermore, the techniques that have so far been proposed to investigate sum-rank-metric codes seem to perform well only in one of the two extreme cases (either small matrix spaces, or few blocks).

Even though 
several mathematical properties of sum-rank-metric codes have been discovered in the past years (see all the references above), it is still widely unknown what the largest cardinality of a sum-rank-metric code with given correction capability can be. This number has been computed in various special cases, most notably when the size of the underlying field size is large and when the matrix spaces are all of the same type. 
Several bounds and constructions of sum-rank-metric codes are available in the literature, most of which are surveyed or established in~\cite{byrne2021fundamental}.

In this paper, we introduce a novel eigenvalue method to establish upper bounds on the cardinality of a sum-rank-metric code with given correction capability. 
The method applies to any sum-rank-metric ambient space over any field. In fact, our results even apply to nonlinear codes.
Our approach uses algebraic graph theory and linear optimization: we construct a graph that naturally encodes the sum-rank metric, and we show that the size of the largest codes 
coincides with a generalization of the independence number of such graph (the so-called $k$-independence number).
We then use results from spectral graph theory and linear programming~\cite{ACF2019,F2020} (in particular, the so-called \textit{Ratio-Type} bound on the $k$-independence number, and a linear optimization method to calculate it) to estimate this quantity and derive the new eigenvalue bounds.

It turns out that our spectral bounds outperform for several instances the best bounds currently available.
Moreover, our bounds can be made explicit for small values of the minimum distance ($d=3$ and $d=4$), and the results indicate very clearly that spectral methods are able to capture the nature of the sum-rank metric better than the known methods. For general values of the minimum distance $d$, a linear optimization method is provided in order to calculate the best value of the general eigenvalue bound.
Our bounds heavily rely on spectral graph theory techniques and do not appear to have any elementary derivation.

As a by-product, this paper also initiates the study of sum-rank-metric graphs and their mathematical properties.  
More precisely, we prove that said graphs are vertex-transitive, $\delta$-regular (where we explicitly compute the value of $\delta$), and walk-regular. Note that some of these properties are
indeed 
needed for our eigenvalue machinery to work. We also characterize when sum-rank-metric graphs are distance-regular. Moreover, we  compute the entire spectrum of sum-rank-metric graphs, which is the key ingredient needed for using our bounds on the cardinality of sum-rank-metric codes.

\paragraph{Outline.} 
This paper is organized as follows. In Section~\ref{sec:2} we define the sum-rank metric and sum-rank-metric codes, surveying the known bounds for their cardinality.
In Section~\ref{sec:3} we introduce and study the combinatorial properties of sum-rank-metric graphs.
In Section~\ref{sec:4} we present our eigenvalue bounds for sum-rank-metric graphs, and in Section~\ref{sec:5} we compare them with the state of the art.
Finally, in Section~\ref{sec:6} we show the implications of our bounds for the block length of MSRD codes, obtaining various
non-existence results for (possibly nonlinear) MSRD codes.


\bigskip

\section{Sum-Rank-Metric Codes}
\label{sec:srmc}
\label{sec:2}

Throughout this paper, 
$q$ denotes a prime power and $\Fq$ is the finite field with $q$ elements.
We let $t$ be a positive integer
and $\bfn=(n_1,\ldots,n_t)$, $\bfm=(m_1,\ldots,m_t) \in~\N^t$ are
ordered tuples with $m_1 \geq m_2 \geq \cdots \geq m_t$ and $m_i\geq n_i$ for all $i\in\{1,\dots,t\}$.
We also put 
$N\coloneqq n_1+\cdots+n_t$ for ease of notation.

\begin{definition}
    The \textbf{sum-rank-metric space} is the direct sum
$$\spac\coloneqq \bigoplus_{i=1}^t \F_q^{n_i \times m_i},$$
which is an $\F_q$-linear space itself.
The \textbf{sum-rank}
of $(X_1,\ldots, X_t)\in\spac$
is 
$\srk(X) = \sum_{i=1}^t \rk(X_i)$.
The \textbf{sum-rank distance} between $X,Y \in \spac$ 
is $\srk(X - Y)$.  
\end{definition}

It is known and very easy to see
that via $(X, Y ) \mapsto \srk(X - Y)$ is indeed a distance, i.e., it is symmetric, positive, meaning that for any $X$ and $Y$ $\srk(X-Y)\geq 0$ with $\srk(X-Y) = 0$ if and only if $X=Y$, and also satisfies the triangular inequality. 

\begin{remark} \label{rem:reduce} Note that the sum-rank distance is a generalization of both rank distance, which is reached for $t=1$, and the Hamming distance, which is reached for $n_i=m_i=1$ for $i\in\{1,\dots,t\}$.
\end{remark}

The key players of this paper are sum-rank-metric codes, defined as follows.

\begin{definition}
A (\textbf{sum-rank-metric}) \textbf{code} is a non-empty
subset
$\mC \subseteq \spac$.
The \textbf{minimum} (\textbf{sum-rank}) \textbf{distance} of  a code $\mC$ with $|\mC| \ge 2$ is defined via $$\srk(\mC)\coloneqq\min\{\srk(X-Y) \st X, Y \in \mC, \, X \neq Y\}.$$
In this context, the elements of $\mC$ are called \textbf{codewords}.
\end{definition}

This paper focuses on the largest cardinality a sum-rank-metric code can have for the given minimum distance. 

\begin{notation}
    We let
    $A_q(\bfn,\bfm,d)$ to denote the largest cardinality of a sum-rank-metric code in $\spac$ with the property that
    any two distinct codewords 
    of $\mC$ are at sum-rank distance at least $d$ from each other.
\end{notation}

In our approach, we will study sum-rank-metric codes using techniques from spectral graph theory.
We thus introduce the following object.

\begin{definition}
    The \textbf{sum-rank-metric graph} is the graph
    whose vertices are the  elements of
$\spac$ and whose edges are the pairs $(X,Y) \in\spac^2$ with the property that  $\srk(X-Y) = 1$.
We denote it by $\Gamma(\spac)$.
\end{definition}

In the case where $t=1$, $\bfn=(n)$, and $\bfm=(m)$, the notation $\Mat(n,m,\F_q)$ and $\Gamma(\Mat(n,m,\F_q))$ will sometimes be used instead of $\Mat(\bfn,\bfm,\F_q)$ and $\Gamma(\Mat(\bfn,\bfm,\F_q))$. Such graph is called \textbf{rank-metric graph}, which is also known in the literature as \textbf{bilinear forms graph} (e.g. see~\cite{BCNbookDRG}). For the rank metric, the Singleton bound is always met~\cite{D1978}, so there has not been a combinatorial investigation of the graphs associated with it.

Even if the sum-rank-metric graph is defined by capturing only the pairs $(X,Y)$ at distance 1, it encodes all distances between pairs of tuples of matrices. This fact follows from the following result.

\begin{proposition}[see~\text{\cite[Proposition 4.3]{byrne2022anticodes}}]
\label{prop:alberto}
The geodesic distance of $\Gamma(\spac)$ coincides with the sum-rank distance.
\end{proposition}

\subsection{Overview of Known Bounds}\label{sec:code-bounds}


In this subsection we
concisely recall the main bounds currently known on sum-rank-metric codes.
To our best knowledge, 
the overview
offered in 
\cite{byrne2021fundamental}
is up to date and we refer to that for more details.
 At the end of the subsection we also present a new bound based on Ramsey theory. While the bound is not expected to be good in general, it shows an interesting connection between coding theory and Ramsey theory.

In~\cite{byrne2021fundamental}, the authors
observe that if $\mC$ is a sum-rank-metric code in $\spac$ of size at least $2$ with $\srk(\mC) \ge d$,
then the size of $\mC$ is upper bounded by the size of the largest Hamming-metric code over field $\F_{q^m}$ of length $N$ and minimum distance at least $d$, where
$m=\max\{m_1,\dots,m_t\}$.
This observation implies the following bounds.

\begin{theorem}[see~\text{\cite[Theorem III.1]{byrne2021fundamental}}]\label{thm:induced} Let $m=\max\{m_1,\dots,m_t\}$ and let $\mC\subseteq\spac$ be a sum-rank-metric code with $|\mC|\geq 2$ and $\srk(\mC) \ge d$. The following hold.
\begin{center}
    \begin{tabular}{ll}
     \textbf{Induced Singleton bound:} & $|\mC|\leq q^{m(N-d+1)}$, \\[0.5cm]
     \textbf{Induced Hamming bound:} & $|\mC|\leq \floor{\frac{q^{mN}}{\sum_{s=0}^{\floor{(d-1)/2}}\binom{N}{s}(q^m-1)^s}}$, \\[0.5cm]
     \textbf{Induced Plotkin bound:} & $|\mC|\leq \floor{\frac{q^{m}d}{q^md-(q^m-1)N}}$ if $d>(q^m-1)N/q^m$, \\[0.5cm]
     \textbf{Induced Elias bound:} & $|\mC|\leq \floor{\frac{Nd(q^m-1)}{q^mw^2-2Nw(q^m-1)+(q^m-1)Nd}\cdot\frac{q^{mN}}{V_w(\F^N_{q^m})}}$. \\[0.5cm]
\end{tabular}
\end{center}
In the Induced Elias bound, $w$ is any integer between $0$ and $N(q^m-1)/q^m$ such that the denominator is positive, and $V_w(\F_{q^m}^N)=\sum_{i=0}^w {N\choose i}(q^m-1)^i$, i.e. the cardinality of any ball of radius $w$ in $\F_{q^m}^N$ with respect to the Hamming distance.
\end{theorem}

The following results are also taken from~\cite{byrne2021fundamental}, but are not induced by Hamming-metric codes.
Each of the bounds has a separate proof, for which the reader is referred to~\cite{byrne2021fundamental}.

\begin{theorem}[see \text{\cite[Theorems III.2, III.6--III.8]{byrne2021fundamental}}]\label{thm:non-induced} Let $\mC\subseteq\spac$ be a code with $|\mC|\geq2$ and $\srk(\mC)\ge d$. Let $j$ and $\delta$ be the unique integers satisfying $d-1=\sum_{i=1}^{j-1} n_i + \delta$ and $0\leq\delta\leq n_j-1$. Let $\ell\leq t-1$ and $\delta'\leq n_{\ell+1}-1$ be the unique positive integers such that $\smash{d-3=\sum_{j=1}^\ell n_j+\delta'}$.
Define $\smash{\bfn'=(n_{\ell+1}-\delta',n_{\ell+2},\dots,n_t)}$ and $\bfm'=(m_{\ell+1},m_{\ell+2},\dots,m_t)$. Finally, let $Q=\sum_{i=1}^t q^{-m_i}$.
The following hold.
\begin{center}
    \begin{tabular}{ll}
     \textbf{Singleton bound:} & $|\mC|\leq q^{\sum_{i=j}^t m_in_i-m_j\delta}$, \\[0.5cm]
     \textbf{Sphere-Packing bound:} & $|\mC|\leq\floor{\frac{|\spac|}{V_r(\spac)}}$, where $r=\floor{(d-1)/2}$, \\[0.5cm]
     \textbf{Projective Sphere-Packing bound:} & $|\mC|\leq\floor{\frac{|\Mat(\bfn',\bfm',\Fq)|}{V_1(\Mat(\bfn',\bfm',\Fq))}}$ if $3\leq d\leq N$, \\[0.5cm]
     \textbf{Total Distance bound:} & $|\mC|\leq\frac{d-N+t}{d-N+Q}$ if $d>N-Q$. \\[0.5cm]
\end{tabular}
\end{center}
In the Sphere-Packing and Projective Sphere-Packing bounds, the denominator
denotes the cardinality of any ball in the sum-rank-metric of radius $r$. For example,
$$V_r(\spac)=|\{(X_1,\dots,X_t)\in\spac \mid \srk(X_1,\dots,X_t)\leq r\}|,$$
and a closed formula is provided in~\cite{byrne2021fundamental}.
\end{theorem}

Using the same connection to the size of the largest Hamming-metric code, we can extend the Ramsey-type bound given in~\cite[Theorem 2.1]{CFHMSV22} to the sum-rank-metric case. Let $R(k;r,s)$ be a \textbf{set-coloring Ramsey number}, which is the minimum positive integer $n$ such that in every edge-coloring $\chi:E(K_n)\to{r\choose s}$, where each edge is mapped onto a set of $s$ colors chosen from a palette of $r$ colors, there exists a monochromatic $k$-clique.

The following theorem establishes the first connection between sum-rank-metric codes and set-coloring Ramsey numbers.

\begin{theorem} Let $m=\max\{m_1,\dots,m_t\}$. Let $\mC\subseteq\spac$ be a sum-rank-metric code with $|\mC|\geq 2$ and $\srk(\mC) \ge d$ with $d<N$. Also, let $a$ and $b$ be positive integers such that $b<a$. Finally, let $r=aN$ and $s=db$. If $q^m=R(k;a,b)-1$
then there exist constants $c,c'>0$ such that, for any integers $k,r,s$ with $k\geq3$ and $r>s\geq1$ we have $$|\mC|<R(k;r,s)\leq 2^{c'k(r-s)^2r^{-1}\log\left(\frac{r}{\min(s,r-s)}\right)}.$$
\end{theorem}
\begin{proof}
Following the same observation as the one used in Theorem~\ref{thm:induced}, the size of $\mC$ is bounded by the size of the largest Hamming-metric code over $\F_{q^m}$ of length $N$. The connection between set-coloring Ramsey numbers and the size of the largest Hamming code used for the lower bound on $R(k;r,s)$ is given in~\cite[Theorem 2.1]{CFHMSV22}. The upper bound on $R(k;r,s)$ is given by~\cite[Theorem 1.1]{CFHMSV22}.
\end{proof}

\section{The Sum-Rank-Metric Graph}
\label{sec:3}

In this section, several structural properties of $\Gamma(\spac)$ are investigated, including regularity, distance-regularity, and walk-regularity. One of the reasons for investigating these, besides mathematical curiosity,
is to identify what kind of existing eigenvalue bounds on the $k$-independence number can be applied to $\Gamma(\spac)$.

As we will see in later sections, these techniques lead to upper bounds for the size of sum-rank-metric codes that improve the best known bounds for several parameters.

\begin{notation}
Throughout the paper, unless otherwise specified, we will use $\lambda_1,\lambda_2,\dots,\lambda_n$, to denote the (not necessarily distinct) eigenvalues of the adjacency matrix of a graph with~$n$ vertices. To denote the $r+1$ \textit{distinct} eigenvalues of the adjacency matrix we will use $\theta_0,\theta_1,\dots,\theta_r$, with multiplicity of $\theta_i$ denoted by $m(\theta_i)$. In general we assume $\lambda_1\geq\lambda_2\geq\cdots\geq\lambda_n$ and $\theta_0>\theta_1>\cdots>\theta_r$, unless explicitly specified otherwise.
The adjacency relation on the vertices of a graph is denoted by $\sim$ and the \textbf{degree} of a vertex is the number of edges that are incident to it. 
\end{notation}

Recall that the \textbf{Cartesian product} of  graphs $G$ and $H$ is the graph $G\times H$ whose
vertices are the elements of
the Cartesian product of the vertices of $G$ and $H$, and where vertices $(g_1,h_1)$ and $(g_2,h_2)$ are adjacent if either $g_1=g_2$ and $h_1\sim h_2$, or $g_1\sim g_2$ and $h_1=h_2$. This definition can be inductively extended to a Cartesian product of any finite number of graphs.

The following result shows that the sum-rank-metric graph is the Cartesian product of bilinear forms graphs.

\begin{proposition}\label{prop:prod}
    Let $\Gamma_i=\Gamma(\Mat(n_i,m_i,\Fq))$ for $i\in \{1, \ldots, t\}$. Then $\Gamma(\spac)$ is the Cartesian product $\Gamma_1\times\dots\times\Gamma_t$. 
\end{proposition}
\begin{proof}
    We prove it by induction on $t$. For $t=1$, the claim is trivial. Suppose $t \ge 2$ and let $X_i$ and $Y_i$ be two matrices from $\Mat(n_i,m_i,\Fq)$ which correspond to two vertices of~$\Gamma_i$ for $i\in \{1, \ldots, t\}$. By the induction hypothesis, $\Gamma':=\Gamma_1\times\cdots\times\Gamma_{t-1}$ is $\Gamma(\Mat(\bfn',\bfm,\Fq))$, where $\smash{\bfn'=(n_1,\dots,n_{t-1})}$ and $\smash{\bfm'=(m_1,\dots,m_{t-1})}$. Consider the Cartesian product $\Gamma_1\times\dots\times\Gamma_{t-1}\times\Gamma_t$. The vertex set of this graph is clearly $\spac$, and two vertices $(X_1,\dots,X_t)$ and $(Y_1,\dots,Y_t)$ are adjacent if and only if one of the following holds:
    \begin{itemize}
        \item $(X_1,\dots,X_{t-1})=(Y_1,\dots,Y_{t-1})$ and $X_t\sim Y_t$ in $\Gamma_t$, which is equivalent to\\ $\srk((X_1,\dots,X_{t-1})-(Y_1,\dots,Y_{t-1}))=0$ and $\rk(X_t,Y_t)=1$;
        \item $(X_1,\dots,X_{t-1})\sim(Y_1,\dots,Y_{t-1})$ in $\Gamma'$ and $X_t=Y_t$, which is equivalent to\\ $\srk((X_1,\dots,X_{t-1})-(Y_1,\dots,Y_{t-1}))=1$ and $\rk(X_t,Y_t)=0$.
    \end{itemize}
    In either case
    it follows that $\srk((X_1,\dots,X_t)-(Y_1,\dots,Y_t))=1$.
    Moreover, 
    these are the only two possibilities when $(X_1,\dots,X_t)$ and $(Y_1,\dots,Y_t)$ are adjacent in $\Gamma(\spac)$. Hence $\Gamma_1\times\dots\times\Gamma_t$ is exactly $\Gamma(\spac)$, as claimed.
\end{proof}
Note that the graphs $\Gamma_i$ defined in Proposition~\ref{prop:prod} are bilinear forms graphs. More information about them can be found, for example, in~\cite[Section 9.5.A]{BCNbookDRG}.

Recall that a graph is \textbf{$a$-regular} if the degree of every vertex is $a$. A graph is called \textbf{vertex-transitive} if for any two vertices $u,v$ there exists a graph automorphism $f$ such that $f(u)=v$. A graph is a \textbf{Cayley graph} over a group $G$ with a connecting set $S\subseteq G$ if the vertices of the graph are the elements of $G$, and two vertices $x,y$ are adjacent if and only if there exists and element $s\in S$ such that $x+s=y$. We assume $S$ does not contain the neutral element of $G$ and is closed under taking inverses so that the Cayley graph is always undirected and without loops. The following result shows that the sum-rank-metric graph is a Cayley graph, which implies that it is also vertex-transitive and hence regular. 

\begin{proposition}\label{prop:reg} The graph $\Gamma(\spac)$ is the Cayley graph over an additive group $\spac$ with a connecting set $S=\{X\in\spac \mid \srk(X)=1\}$. In particular, $\Gamma(\spac)$ is vertex-transitive and $\delta$-regular with $\delta=\frac1{q-1}\sum_{i=1}^t (q^{n_i}-1)(q^{m_i}-1).$
\end{proposition}
\begin{proof} By definition of the sum-rank-metric graph, if two vertices $A$ and $B$ are adjacent, then $\srk(A-B)=1$. Clearly $A-B\in S$, the connecting set of the Cayley graph. On the other hand, for an arbitrary $X$ such that $X\in S$, the vertices $A$ and $A+X$ are adjacent since $\srk(A+X-A)=\srk(X)=1$. Thus, the graph is indeed Cayley, and hence vertex-transitive.

The regularity parameter $\delta$, which is also the largest eigenvalue of the spectrum of the graph, can be calculated because it is
the cardinality of the connecting set $S$, i.e., the neighborhood of $O$. Since $\srk(X)=1$ implies that there exists a single $i\in\{1,\dots,t\}$ such that $\rk(X_i)=1$ while $\rk(X_j)=0$ for all $j\neq i$, we have
$$\delta = \frac1{q-1}\sum\limits_{i=1}^t (q^{n_i}-1)(q^{m_i}-1)$$
by~\cite[Theorem 9.5.2]{BCNbookDRG}.
\end{proof} 

We now turn to the (partial) walk-regularity of
sum-rank-metric graphs.
Recall that a graph is
\textbf{$l$-partially walk-regular} if for any vertex $v$ and any positive integer $i\leq l$ the number of closed walks of length~$i$ that start and end in~$v$ does not depend on the choice of~$v$. A graph is \textbf{walk-regular} if it is $l$-partially walk-regular for any positive integer~$l$.

As already mentioned, the main motivation for studying the walk-regularity of the sum-rank-metric graph is to check the applicability of spectral methods
to control the
graph's $k$-independence number,
which we will present later in Section~\ref{sec:boundsss}.

\begin{proposition}\label{prop:walk-regular}
    The graph $\Gamma(\spac)$ is walk-regular.
\end{proposition}
\begin{proof}
    The claim immediately follows from the observation that, for any fixed positive integer $k$, there is a bijection between a set of closed $k$-walks that start at $O$ (which is the vertex corresponding to the tuple of zero metrices) and a set of closed $k$-walks that start any other vertex
    $v= (X_1,X_2,\dots,X_t)$. The bijection from walks starting with $O$ to walks starting with $v$ is induced by the bijection 
    of $\spac$ given by $(Y_1,Y_2,\dots,Y_t)\mapsto (Y_1+X_1,Y_2+X_2,\dots,Y_t+X_t)$.
\end{proof}

Another important property of a graph is (partial) distance regularity.
For two vertices $x$ and $y$ at distance $i$ from each other, let $p^i_{j,h}(x,y)$ denote the number of vertices that are both at distance $j$ from $x$ and at distance $h$ from $y$. We say that a graph is \textbf{$l$-partially distance-regular} if for any integers $i,j,h$ such that $j,h\leq l$ and $i\leq j+h\leq l$ the value $p^i_{j,h}(x,y)$ does not depend on the choice of the vertices $x$ and $y$.
In particular, a graph is $l$-partially distance-regular if the values $\smash{c_i(x,y)=p^i_{1,i-1}(x,y)}$, $\smash{a_{i-1}(x,y)=p^{i-1}_{1,i-1}(x,y)}$, $\smash{b_{i-2}(x,y)=p^{i-2}_{1,i-1}(x,y)}$ only depend on $i\leq l$ and not on the choice of the vertices $x,y$. If a graph is $l$-partially distance-regular for any $l$, then it is called \textbf{distance-regular}. In a distance-regular graph $\Gamma$ with geodesic distance $d$, the values $a_i=a_i(x,y)$, $b_i=b_i(x,y)$, and $c_i=c_i(x,y)$ for $i=0,\dots, D=\max_{x,y\in \Gamma}d(x,y)$ form the \textbf{intersection array} of the graph. Because distance-regular graphs are also $\delta$-regular for some $\delta$, we have $a_i+b_i+c_i=0$ for any suitable~$i$, and in particular, $b_0=\delta$ and $a_0=c_0=0$. This implies that, to define the intersection array of the distance-regular graph, it is sufficient to define the values $b_i$ and $c_i$. Note that any $l$-partially distance-regular graph is also $l$-partially walk-regular, but the converse is not true in general. For instance, we will see in Proposition~\ref{prop:drg} that the sum-rank-metric graphs, despite being walk-regular (Proposition~\ref{prop:walk-regular}), are not even necessarily $2$-partially distance-regular.

The next result characterizes for which parameters
sum-rank-metric graphs are (partially) distance-regular.

\begin{proposition}\label{prop:drg}
    Suppose $t\geq2$. Then the graph $\Gamma(\spac)$ is distance-regular if and only if $n_i=1$ and $m_i=m$ for all $i\in \{1, \ldots, t\}$. The intersection array in that case is given by $b_i=(t-i)(q^m-1)$ and $c_i=i$ for all $i \in \{0,\dots,t\}$. Moreover, if $\Gamma(\spac)$ is not distance-regular, then it is also not $2$-partially distance-regular.
\end{proposition}
\begin{proof}
$(\Leftarrow)$ To prove that $\Gamma(\spac)$ is distance-regular, it is enough to show that for vertices $v$ and $u$ with $d(v,u)=i$ the values $a_i$, $b_i$, and $c_i$ (which are the number of neighbors of $v$ at distance $i$, $i+1$, and $i-1$ from $u$, respectively), do not depend on the choice of $v$ and $u$. Since $\Gamma(\spac)$ is vertex-transitive by Proposition~\ref{prop:reg}, we may assume without loss of generality that $u$ corresponds to $O$, the element of $t$ zero matrices of size $1\times m$. Each matrix of size $1\times m$ has rank at most $1$. Since $d(u,v)=i$, the element corresponding to $v$ is a collection of $i$ non-zero and~$t-i$ zero matrices.

To obtain the number of neighbors of $v$ at distance~$i$ from $O$ we choose one of the~$i$ non-zero matrices of~$v$ and substitute it for a different non-zero $1\times m$ matrix. There are~$q^m-2$ possibilities for this, so~$a_i=i(q^m-2)$.
For the number of neighbors of~$v$ at distance~$i+1$ from~$O$, we take one of the~$t-i$ zero matrices of~$v$ and replace it with one of~$q^m-1$ non-zero matrices of size $1\times m$. Hence $b_i=(t-i)(q^m-1)$.
To get a neighbor of $v$ at distance $i-1$ from $O$, we choose one of the $i$ non-zero matrices of $v$ and replace it with a zero matrix, which means $c_i=i$.

It is easy to verify that $a_i+b_i+c_i=\delta$, the regularity of $\Gamma(\spac)$ given by Proposition~\ref{prop:reg}. Moreover, neither of these values depends on the choice of $v$, as required.

$(\Rightarrow)$ Next, we show that $n_i=1$ and $m_i=m$ is the only case when $\Gamma(\spac)$ is distance-regular. For this, we consider the following two possibilities separately.

First, suppose $n_i=1$ for all $i\in\{1,\dots,t\}$ but $m_i\neq m_j$ for some $i,j\in \{1,\dots,t\}$. As before, we attempt to calculate $a_i(v,O)$, where $O$ is the vertex corresponding to an element of all-zero matrices. Without loss of generality assume that $m_i\in \{m,m'\}$ for all $i\in\{1,\dots,t\}$ and $m\neq m'$. Then $a_i(v,O)$ is given by $j(q^m-1)+k(q^{m'}-1)$, where $j$ is the number of non-zero matrices of size $1\times m$ in $v$ and $k$ is the number of non-zero matrices of size $1\times m'$ in $v$. Clearly $j+k=i$, but the exact values of both $j$ and $k$ depend on the choice of $v$, so the value $a_i$ cannot be defined from~$i$ alone. Note that, in particular, $a_2(v,O)$ depends on the choice of $v$, so in this case the graph is also not $2$-partially distance-regular.

Next, we consider the possibility that $n_j\geq 2$ for some $j\in\{1,\dots,t\}$, which implies that now some elements of $\spac$ include matrices of rank $2$. Consider two vertices $v$ and $v'$ of the graph $\Gamma(\spac)$, the vertex $v$ is an element of $\spac$ with $i$ matrices of rank $1$ and $t-i$ matrices of rank $0$, while $v'$ is an element with one matrix $X_j$ of rank $2$, $i-2$ matrices of rank $1$, and $t-i+2$ matrices of rank $0$. Both these elements have sum-rank equal to $i$, which means that both vertices $v$ and $v'$ are at distance $i$ from~$O$.

We now attempt to calculate $c_i(v,O)$ and $c_i(v',O)$. For the vertex $v$, $c_i(v,O)=i$ by the same argument as the one presented in the ``if'' part of this proof. For the vertex $v'$, we apply a similar argument to obtain $c_i(v',O)=i-2+k$, where $k$ denotes the number of matrices $M$ of size $n_j\times m_j$ such that $\rk(M)=1$ and $\rk(X_j-M)=1$. In order for $\Gamma(\spac)$ to be a distance-regular graph, $k$ must be exactly $2$. However, it is easy to construct three different matrices that satisfy the requirement. Indeed, suppose $\bf{x_1}$ and $\bf{x_2}$ are two linearly independent non-zero vectors such that each row of {$X_j$} can be presented as $\alpha_i \bf{x_1}+\beta_i \bf{x_2}$ for $i\in \{1, \ldots, {n_j}\}$. Suppose that $M_1$, $M_2$ and $M_3$ are {$n_j\times m_j$} matrices such that the rows of $M_1$ are $\alpha_i \bf{x_1}$, the rows of $M_2$ are $\beta_i \bf{x_2}$, and the rows of $M_3$ are $\alpha_i (\bf{x_1}-\bf{x_2})$ for $i\in \{1, \ldots, n_1\}$. It is easy to see that $\rk(X_j-M_i)=1$ and $\rk(M_i)=1$ for $i\in\{1,2,3\}$. Hence, $c_i(x,O)$ depends on the choice of vertex $x$, which contradicts the distance-regularity of the graph. Note that, in particular, $c_2(x,O)$ depends on the choice of $x$, so in the case where $n_1\geq 2$ the graph $\Gamma(\spac)$ is also not $2$-partially distance-regular.
\end{proof}

Note that in the case $n_i=1$ and $m_i=m$ 
for all $i\in {\{1,\dots,t\}}$
and some $m$, the space $\spac$ with the sum-rank distance is $\F_{q}^m$ with the Hamming distance. In other words, when $\Gamma(\spac)$ is a distance-regular graph, then it is a Hamming-metric graph.

We also recall that a graph is \textbf{distance-transitive} if for any $i$ and any two pairs of vertices $v,w$ and $x,y$ such that $d(v,w)=d(x,y)=i$ there exists an automorphism of the graph that maps $v$ to $x$ and $w$ to $y$. It is well known that distance-transitive graphs are also distance-regular. In particular, sum-rank-metric graphs cannot be distance-transitive in general by Proposition~\ref{prop:drg}.
The next two lemmas treat the case $t=1$, where sum-rank-metric graphs reduce to graphs of bilinear forms.

\begin{lemma}
Suppose $t=1$. Then the graph $\Gamma(\Mat(n,m,\Fq))$ is distance-transitive with diameter $n$ and the intersection array given by  \begin{align*}
    b_i &= \frac1{q-1}q^{2i}(q^{m-i}-1)(q^{n-i}-1), \\
    c_i &= \frac1{q-1} q^{i-1} (q^i-1),
\end{align*}
{for all $i \in \{0, \ldots, n\}$.}
\end{lemma}
\begin{proof}
    This is an immediate consequence of \cite[Theorem 9.5.2]{BCNbookDRG}.
\end{proof}

We say that a distance-regular graph has \textbf{classical parameters} $({D},q,\alpha,\beta)$ if it has diameter {$D$} and intersection array given by
$$b_i=\frac{q^{{D}}-q^i}{q-1}\left(\beta-\alpha\frac{q^i-1}{q-1}\right), \quad c_i=\frac{q^i-1}{q-1}\left(1+\alpha\frac{q^{i-1}-1}{q-1}\right),$$
{for all $i \in \{0, \ldots, D\}$.}
Note that the graph $\Gamma(\Mat(n,m,\Fq))$ is a distance-regular graph with classical parameters $(n,q,q-1,q^m-1)$, whose eigenvalues are given in the following result.

\begin{lemma}\label{lem:BFG_evals} Suppose $t=1$. Then the graph $\Gamma(\Mat(n,m,\Fq))$   has $n+1$ distinct eigenvalues given by
$$\theta_i=\frac{(q^{n-i}-1)(q^m-q^i)-q^i+1}{q-1},\quad i=0,\dots,n,$$
with respective multiplicities given by
$$m(\theta_i)=\prod\limits_{s=0}^{i-1}\frac{(q^{n-s}-1)(q^m-q^s)}{q^{s+1}-1}.$$
\end{lemma}
\begin{proof}
    This is an immediate consequence of \cite[Corollary 8.4.2]{BCNbookDRG}, which provides the eigenvalues, and of \cite[Theorem 8.4.3]{BCNbookDRG}, which gives the multiplicities.
\end{proof}

Proposition~\ref{prop:prod} combined with Lemma~\ref{lem:BFG_evals} and classical graph theory results
give us a complete description of the spectrum of the sum-rank-metric graph $\Gamma(\spac)$.

\begin{proposition}\label{prop:evals}
The graph $\Gamma(\spac)$ has the eigenvalue
$$\lambda_{\bf i} = \sum\limits_{j=1}^t\frac{(q^{n_j-i_j}-1)(q^{m_j}-q^{i_j})-q^{i_j}+1}{q-1},$$
for every tuple ${\bf i}=(i_1,\dots,i_t)$ such that $i_j \in \{0, \ldots, n_j\}$ where $j\in\{1, \ldots, t\}$, and no other eigenvalues.
The multiplicity of $\lambda_{\bf i}$ is given by
$$m(\lambda_{\bf i}) = \sum_{{\bf i'}: \lambda_{\bf i'} = \lambda_{\bf i}} \prod\limits_{j=1}^t\prod\limits_{s=0}^{i'_j-1}\frac{(q^{n_j-s}-1)(q^m_j-q^s)}{q^{s+1}-1},$$
where ${\bf i'}=(i'_1,\dots,i'_t)$ is a tuple such that $i'_j\in\{0,\dots,n_j\}$ and $\lambda_{\bf i} = \lambda_{\bf i'}$.
\end{proposition}
\begin{proof}
    It is a well-known fact that, if $\Gamma_1$ and $\Gamma_2$ are graphs with distinct eigenvalues $\lambda_i$, $i\in \{1, \ldots, n_1\}$, and $\mu_j$, $j\in\{1, \ldots, n_2\}$, then the eigenvalues of their Cartesian product $\Gamma_1\times\Gamma_2$  have form $\lambda_i+\mu_j$ for $(i,j) \in \{1, \ldots, n_1\} \times \{1, \ldots, n_2\}$; see for instance \cite{CDSbook}). This can be inductively extended to a Cartesian product of $t\geq2$ graphs. Combining this fact with Proposition~\ref{prop:prod} and Lemma~\ref{lem:BFG_evals}, we obtain the desired result.
\end{proof}


\section{New Eigenvalue Bounds 
for Sum-Rank-Metric Codes}
\label{sec:boundsss}
\label{sec:4}

In graph theory, the \textbf{$k$-independence number} of a graph $G$, denoted by $\alpha_k$, is the size of the largest set of vertices such that any two vertices in the set are at a distance greater than~$k$ from each other. Alternatively, one can consider the \textbf{$k$-th power graph} $G^k$ with the same vertex set as $G$ such that two vertices in $G^k$ are adjacent if and only if the distance between these two vertices in $G$ is at most~$k$. Then the $k$-independence number of $G$ is also the $1$-independence number of $G^k$, or the largest size of an independent set in $G^k$. Note that, to apply spectral bounds on the $1$-independence number of the $k$-th power graph, one needs to know how the spectrum of $G^k$ relates to the spectrum of $G$. In general, this relation is not known.

In~\cite{ACF2019}, sharp eigenvalue upper bounds on the $k$-independence number $\alpha_k$ of a graph were derived. In particular, Ratio-Type and Inertia-Type bounds were presented, which generalize the well-known Cvetkovi\'c's~\cite{cvetkovic1971} and Hoffman's bounds (unpublished, see e.g. \cite{h2021}) for the classic $1$-independence number.  Both bounds depend on a choice of a real polynomial $p$ of degree $k$ that can be found using mixed-integer linear programming (MILP). In~\cite{ACFNZ2022},
the authors present a MILP implementation for the Inertia-Type bound for general graphs. In~\cite{F2020}, an LP approach was proposed to find the best Ratio-Type bound for partially walk-regular graphs using minor polynomials, which has a connection to Delsarte's LP method. In particular, it was shown in the same paper~\cite{F2020} that the LP bound on certain families of Hamming graphs coincides with Delsarte's LP bound (more details on Delsarte's bound are given later on in this section).
In \cite{F2020}, a Ratio-Type bound
was used to show the non-existence of 1-perfect codes in odd graphs.

The following result establishes a link between coding theory and graph theory. It will play a key role in establishing our main results in Section~\ref{sec:new}.

\begin{corollary}\label{coro:equivalence}
$A_q(\bfn,\bfm,d)=\alpha_{d-1}\big(\Gamma(\spac)\big)$.
\end{corollary}

\begin{proof}
    It follows from Proposition~\ref{prop:alberto} that $A_q(\bfn,\bfm,d)=\alpha_{d-1}$, where $\alpha_{d-1}$ is the $(d-1)$-independence number of $\Gamma(\spac)$. Indeed, let $S$ be a $(d-1)$-independent set in $V(\Gamma(\spac))$ of the largest cardinality. By definition, the geodesic distance between any two vertices in $S$ is at least~$d$. This implies that $S$ is the largest subset of elements of $\spac$ at distance at least $d$ from each other, which makes it a sum-rank-metric code having minimum distance $d$ with the largest cardinality in $\spac$. Hence $A_q(\bfn,\bfm,d)=|S|=\alpha_{d-1}$. 
\end{proof}

It follows from Corollary~\ref{coro:equivalence} that any upper bound on the $k$-independence number of a graph yields an upper bound for sum-rank-metric codes, and vice versa. The next result provides the main bound we will be using in this paper, the so-called Ratio-Type bound, which we will later apply to sum-rank-metric graphs.

\begin{theorem}[Ratio-Type bound; see \cite{ACF2019}]\label{thm:hoffman-like} Let $G=(V,E)$ be a regular graph with $n$ vertices and adjacency matrix $A$ with eigenvalues $\lambda_1\geq\dots\geq\lambda_n$. Let $p\in\mathbb{R}_k[x]$. Define $W(p)=\max_{u\in V}\{(p(A))_{uu}\}$ and $\lambda(p)=\min_{i=2,\dots,n}\{p(\lambda_i)\}$. Then $$\alpha_k(G) \leq n \, \frac{W(p)-\lambda(p)}{p(\lambda_1)-\lambda(p)}.$$
\end{theorem}

For a fixed $k$, the challenge behind applying the
Ratio-Type bound is
to find a polynomial of degree $k$ which minimizes the right-hand of the inequality above. In \cite{F2020}, an LP bound using minor polynomials was proposed for finding, for a given $k$ and a $k$-partially walk-regular graph $G$, the polynomial $p$ that optimizes the above Ratio-Type bound that we will be using. In this LP, the inputs are the distinct eigenvalues of the graph $G$, denoted $\theta_0>\cdots>\theta_r$, with respective multiplicities $m(\theta_i)$, $i\in\{0,\dots,r\}$. The \textbf{minor polynomial} $f_k\in\R_k[x]$ is the polynomial that minimizes $\sum_{i=0}^r m(\theta_i) f_k(\theta_i)$. Let $p=f_k$ be defined by $f_k(\theta_0)=x_0=1$ and $f_k(\theta_i)=x_i$ for $i\in\{1,\dots,r\}$, where the vector $(x_1,\dots,x_r)$ is a solution of
the following linear program: 
\begin{equation*}
\boxed{
\begin{array}{ll@{}ll}
\text{minimize}  &\sum_{i=0}^{r} m(\theta_i)x_i &\\
\text{subject to} &f[\theta_0,\dots,\theta_s]=0, & \quad s=k+1,\dots,r\\
&x_i\geq0, &\quad i=1,\dots,r\\
\end{array}
}
\end{equation*}
Here, $f[\theta_0,\dots,\theta_m]$ denote $m$-th divided differences of Newton interpolation, recursively defined by $$f[\theta_i,\dots,\theta_j]=\frac{f[\theta_{i+1},\dots,\theta_j]-f[\theta_i,\dots,\theta_{j-1}]}{\theta_j-\theta_i},$$ where $j>i$, starting with $f[\theta_i]=x_i$ for $i\in\{0,\dots,r\}$.
In~\cite{F2020} it was shown that for $k=3$ it is possible to obtain tight bounds for every Hamming graph $H(2,r)$ using this LP approach with minor polynomials, and it is also shown that in for these graphs the Ratio-Type bound coincides with Delsarte's LP bound~\cite{DelsarteLP}.
\begin{remark}
\label{noDels}
We should note that, in order to be applied, Delsarte's LP bound  \cite{DelsarteLP} requires a symmetric association scheme, which is a special partition of $\spac\times\spac$ into $n+1$ binary relations $R_0,R_1,\dots,R_n\in\mathcal{R}$. As part of the definition, for any $R_i,R_j,R_k\in\mathcal{R}$ the number of $z\in\spac$ such that $(x,z)\in R_i$ and $(z,y)\in R_j$ is the same for any $x,y$ such that $(x,y)\in R_k$. Assuming the association scheme is the one naturally induced by
the sum-rank distance (that is, $(x.y)\in R_i$ if and only if $\srk(x-y)=i$ for all $i$) this condition implies that the graph $\Gamma(\spac)$ must be distance-regular.
By Proposition~\ref{prop:drg} and Remark~\ref{rem:reduce}, this is the case only when the sum-rank metric reduces to the rank or to the Hamming metric, which is a classical setting for the Delsarte's LP bound.
In the general case, when $\Gamma(\spac)$ is not distance-regular (for a specific example, we refer to~\cite[Example 4.4]{byrne2022anticodes}), one can still define an association scheme (or, more generally, a coherent configuration)
by ``refining'' the sum-rank distance and then apply Delsarte's LP bound to this scheme.
This will be studied in detail by the authors elsewhere.
\end{remark}

Next we will look at formul{\ae} for the Ratio-Type bound for small~$k$ and any regular graph $G$. In the case where $k=1$, $\alpha_1$ is exactly the independence number and by choosing $p(x)=x$ we get the classical Hoffman's bound (unpublished, see e.g. \cite{h2021}). In the particular cases where $k=2$ and $k=3$ (corresponding to $d=3$ and $d=4$, respectively), the best polynomials that optimize the Ratio-Type bound from Theorem~\ref{thm:hoffman-like} have been found, resulting in the following more explicit bounds.

\begin{theorem}[Ratio-Type bound, $d=3$; see~\cite{ACF2019}]\label{thm:hoffman-k=2} 
Let $G$ be a regular graph with $n$ vertices and distinct eigenvalues of the adjacency matrix $\theta_0>\theta_1>\dots>\theta_r$ with $r\geq2$. Let $\theta_i$ be the largest eigenvalue such that $\theta_i\leq -1$. Then $$\alpha_2(G)\leq n\frac{\theta_0+\theta_i\theta_{i-1}}{(\theta_0-\theta_i)(\theta_0-\theta_{i-1})}.$$ Moreover, this is the best possible bound that can be obtained by choosing a polynomial via Theorem~\ref{thm:hoffman-like}.    
\end{theorem}

The polynomial giving the best Ratio-Type bound bound has also been found in the case 
$d=4$, with a whole paper devoted to the derivation of the polynomial~\cite{KN2022}. This gives an indication of the difficulty of making the Ratio-Type bound explicit.

\begin{theorem}[Ratio-Type bound, $d=4$;
see~\cite{KN2022}]
\label{thm:hoffman-k=3} Let $G$ be a regular graph with $n$ vertices and distinct eigenvalues of the adjacency matrix $\theta_0>\theta_1>\dots>\theta_r$ with $r\geq3$. Let $s$ be the largest index such that $\smash{\theta_s\geq -\frac{\theta_0^2+\theta_0\theta_r-\Delta}{\theta_0(\theta_r+1)}}$, where $\Delta=\max_{u\in V}\{(A^3)_{uu}\}$. Then
$$\alpha_3(G)\leq n\frac{\Delta-\theta_0(\theta_s+\theta_{s+1}+\theta_r)-\theta_s\theta_{s+1}\theta_r}{(\theta_0-\theta_s)(\theta_0-\theta_{s+1})(\theta_0-\theta_r)}.$$ Moreover, this is the best possible bound that can be obtained by choosing a polynomial via Theorem~\ref{thm:hoffman-like}.
\end{theorem}

\begin{remark}
    In this paper we mainly focus on the applications of the previous two theorems. These only allow us to impose that our sum-rank-metric codes have minimum distance bounded from below by 3 or 4. While these codes are not particularly interesting from an application viewpoint, the goal of our paper is give evidence that spectral graph theory captures the sum-rank metric better than the currently available methods.
    Note that the analogues of Theorems~\ref{thm:hoffman-k=2} and~\ref{thm:hoffman-k=3} for larger $d$, or even some weaker versions of such results, would immediately allow studying codes with larger minimum distance via our approach. We leave that to future (graph theory) research.
\end{remark}

The bounds of Theorems~\ref{thm:hoffman-k=2} and~\ref{thm:hoffman-k=3} have been applied to several families of Hamming graphs, which are a special case of sum-rank-metric graphs due to Remark~\ref{rem:reduce}, see \cite{F2020,KN2022} for more details on the Hamming metric.

Recall that in Proposition~\ref{prop:walk-regular} we showed that the graph $\Gamma(\spac)$ is walk-regular. It follows from the definition of walk-regularity that, if $A$ is the adjacency matrix of a walk-regular graph on $n$ vertices, and $p$ is a polynomial with real coefficients, then the diagonal of $p(A)$ is constant with entries $\frac1n \operatorname{tr}(p(A))$, where $\operatorname{tr}(p(A))$ is the trace of $p(A)$. This means in particular that $\Delta$ in Theorem~\ref{thm:hoffman-k=3} can be calculated as $\Delta = \frac1n \sum_{i=0}^r \theta_i^3 m(\theta_i)$, where $m(\theta_i)$ denotes the multiplicity of the eigenvalue $\theta_i$ of $A$.

In~\cite{AK2022} the authors pointed out a connection between $\alpha_k(G_i)$ for some positive integer $k$ and $i\in\{1,\dots,t\}$ and the $k$-independence number of the Cartesian product $G_1\times\cdots\times G_t$. Namely, when extended by induction from a Cartesian product of two graphs to a Cartesian product of $t$ graphs, we obtain:
\begin{lemma}[see~\text{\cite[Theorem 7]{AK2022}}]\label{lem:alphak_cartesian} For any $t$ graphs $G_1,G_2,\dots,G_t$, 
\begin{align*}
    \alpha_k(G_1\times \dots\times G_t) \leq \min&\left(\alpha_k(G_1)\cdot|V_2|\cdots|V_t|,\alpha_k(G_2)\cdot|V_1|\cdot|V_3|\cdots|V_t|,\dots,\right. \\
    &\left.\alpha_k(G_t)\cdot|V_1|\cdots|V_{t-1}|\right).
\end{align*}
where $V_i$ is the vertex set of $G_i$ for $i\in\{1,\dots,t\}$. \end{lemma}
We know from Proposition~\ref{prop:prod} that the sum-rank-metric graph $\Gamma(\spac)$ is the Cartesian product of respective bilinear forms graphs, so Lemma~\ref{lem:alphak_cartesian} can be applied to them. We can also use the fact that the cardinality of the vertex set $V_i$ of a bilinear forms graph $\Gamma(n_i,m_i,\Fq)$ for some integers $n_i,m_i$ is given by $|V_i|=q^{n_im_i}$. In this case, the bound of Lemma~\ref{lem:alphak_cartesian} becomes $$\alpha_k(\Gamma(\spac))\leq \alpha_k(\Gamma(n_l,m_l,\Fq))\cdot q^{\sum_{i=1}^t n_im_i - m_ln_l},$$
where $l\in\{1,\dots,t\}$ is such that $\alpha_k(\Gamma(n_l,m_l,\Fq))\cdot q^{\sum_{i=1}^t n_im_i - m_ln_l}$ is minimal. By comparing this to the Singleton bound of Theorem~\ref{thm:non-induced} one can conclude that the latter would be hard to improve using Lemma~\ref{lem:alphak_cartesian}. Indeed, from a computational check of sum-rank-metric graphs on up to $1000$ vertices, we have not found any instances where the bound of Lemma~\ref{lem:alphak_cartesian} performs better than the bounds of Theorems~\ref{thm:induced} and~\ref{thm:non-induced} as well as the Ratio-Type bound of Theorem~\ref{thm:hoffman-like} when it is applied to the sum-rank-metric graph itself without utilizing its Cartesian product structure.

\begin{remark}  \label{rem:theta}
\begin{enumerate}
    \item A useful upper bound on the $k$-independence number of a graph $G$ is the \textit{Lov\'asz theta number} of $G^k$, denoted by $\vartheta_k$. Recall that $G^k$ denotes the $k$-th power of $G$, which is the graph with the same vertex set at $G$, and where two vertices $u$ and $v$ in $G^k$ are adjacent if and only if the distance between $u$ and $v$ in $G$ is at most $k$. The value of $\vartheta_k$ usually provides a good bound on $\alpha_k$, and it can be estimated using Semidefinite Programming (SDP) as follows~\cite{Lovasz}: For a graph $G$ on $n$ vertices, let $A=(a_{ij})$ range over all $n\times n$ symmetric matrices such that $a_{ii}=1$ for any $i$ and $a_{ij}=1$ for any distinct $i,j$ such that respective vertices of $G$ are non-adjacent. Then $\vartheta = \min_A\lambda_{\max}(A)$, where $\lambda_{\max}(A)$ is the largest eigenvalue of $A$.
    In fact, it was shown in~\cite{S79} that, for graphs derived from symmetric association schemes, the Lov\'asz theta number coincides with the bound obtained through Delsarte's LP method~\cite{DelsarteLP}. Tables~\ref{tab:d=3} and~\ref{tab:d=4} contain examples of sum-rank-metric graphs for which the Lov\'asz theta number $\vartheta_k$, the Ratio-Type bound, and the bounds from Theorems~\ref{thm:induced} and~\ref{thm:non-induced} are calculated. Note that the 
    Lov\'asz theta bound is never worse than the Ratio-Type bound. However, in practice, the Ratio-Type bound
    can be calculated in an exact and quicker way, while the Lov\'asz theta number provides an approximation. Indeed, in order to use the Ratio-Type bound, we only require to know the 
    adjacency spectrum, which we know for our graph of interest by Proposition~\ref{prop:evals}. In contrast, Lov\'asz theta number $\vartheta_k$ requires of the whole graph adjacency matrix in order to be calculated.  
    
    \item Other graph theoretical bounds on the $k$-independence number are known in the literature. For example, \cite{FIRBY199727} introduces a bound on $\alpha_k$ in terms of the average distance in a connected graph. In~\cite{shi2021sharp}, the authors show upper bounds which use the minimum and the maximum degree of the graph, along with infinitely many regular graphs for which these bounds are attained. However, there is no example of a non-distance-regular sum-rank-metric graph up to $1000$ vertices for which these graph theoretical bounds perform better than the Ratio-Type bound of Theorem~\ref{thm:hoffman-like}.
\end{enumerate}
\end{remark}


\section{Explicit Bounds and Comparisons}\label{sec:new}
\label{sec:5}

In this section we turn to concrete applications of our eigenvalue bounds, comparing them with some of the known bounds on sum-rank-metric codes and giving evidence that they improve them. At the end of the section we provide some tables showing that our eigenvalue bounds improve on the best currently known bound.

\begin{theorem}\label{thm:rt-bound_family}
    Let $G= \Gamma(\spac)$ with $\bfn=(n,1,\dots,1)$ and $\bfm=(m,1,\dots,1)$,
    with $m\geq n$ for some integers $m\geq2$ and $n\geq1$. Then $\alpha_2(G)$ is upper bounded by
    {\footnotesize{
\begin{equation}\label{eq:alpha2_bound}
    \frac{q^{m n+t-1}(q-1) \left((q-1)(\varepsilon+1) (\varepsilon-q+1)+\left(q^m-1\right) \left(q^n-1\right)+(q-1)^2 (t-1)\right)}{\left(\varepsilon(q-1)+\left(q^m-1\right)
   \left(q^n-1\right)+(q-1)^2 (t-1)+1\right) \left(\varepsilon(q-1)+\left(q^m-1\right) \left(q^n-1\right)+(q-1)^2 (t-2)\right)}
\end{equation}
}}
where $\varepsilon=(t-1)\mod q$.
\end{theorem}

\begin{proof}
We first note that $\Gamma(\spac)$ is the Cartesian product of $\Gamma(\Mat(n,m,\Fq))$ and the Hamming graph $H(t-1,q)$ on $q^{t-1}$ vertices. The eigenvalues of $\Gamma(\Mat(n,m,\Fq))$ are given by Lemma~\ref{lem:BFG_evals}. The eigenvalues of the Hamming graph are well-known and defined by $(q-1)(t-1)-qj$ for $j=0,\dots,t-1$. Hence by Proposition~\ref{prop:evals} the eigenvalues of $\Gamma(\spac)$ are
$$\lambda_i+(q-1)(t-1)-qj \text{ for } i=0,\dots,n \text{ and } j=0,\dots,t-1,$$ where $\lambda_i=\frac{(q^{n-i}-1)(q^m-q^i)-q^i+1}{q-1}$.

Let $\theta_i$ and $\theta_{i-1}$ be eigenvalues as in Theorem~\ref{thm:hoffman-k=2}, meaning $\theta_i$ is the largest eigenvalue such that $\theta_i\leq-1$ while $\theta_{i-1}$ is the smallest eigenvalue such that $\theta_{i-1}>-1$. Consider the spectrum of $\Gamma(\spac)$ described above. When fixing $\lambda_i$, we have $t$ distinct eigenvalues, and the two eigenvalues among them that are closest to $-1$ are $$\lambda_i+(q-1)(t-1)-q\ceil{\frac{\lambda_i+(q-1)(t-1)+1}{q}}$$ and $$ \lambda_i+(q-1)(t-1)+q-q\ceil{\frac{\lambda_i+(q-1)(t-1)+1}{q}}.$$
Note that, for any integers $a$ and $b>0$, the expression $b\ceil{\frac{a}b}$ is equal to $a$ if $a\mod b\equiv 0$ or to $a+b-(a\mod b)$ otherwise. Observe that $\lambda_i\equiv q-1\mod q$, and hence $\lambda_i+(q-1)(t-1)+1\equiv -(t-1)\mod q$. Let $\varepsilon=(t-1)\mod q$.  Then the two eigenvalues above can be rewritten as: $$-1-\varepsilon$$ and $$q-1-\varepsilon.$$
Note that the number of vertices of $\Gamma(\spac)$ is $q^{mn+t-1}$, and it is a $\delta$-regular graph with $\delta=\frac{(q^n-1)(q^m-1)}{q-1}+(t-1)(q-1)$. Then the Ratio-Type bound on $\alpha_2$ from Theorem~\ref{thm:hoffman-k=2} gives the desired result.
\end{proof}


We now apply the bound of Theorem~\ref{thm:rt-bound_family} in two special cases where we can carefully compare our results with the state of the art.

\subsection{The case \texorpdfstring{$\bfn=(1,\dots,1)$}{n=(1,...,1)}, \texorpdfstring{$\bfm=(m,1,\dots,1)$}{m=(m,1,...,1)}, and \texorpdfstring{$d=3$}{d=3}}\label{subsection:n=1}

We consider the ambient space $\spac$ with $\bfn=(1,\dots,1)$ and $\bfm=(m,1,\dots,1)$ for an arbitrary $t\geq2$ and $m\geq2$. In this case, the graph $\Gamma(\spac)$ is the Cartesian product of the complete graph on $q^m$ vertices $K_{q^m}$
and the Hamming graph $H(t-1,q)$ on $q^{(t-1)}$ vertices.

From the Ratio-Type bound in~(\ref{eq:alpha2_bound}) we derive the bound
\begin{equation}\label{eq:alpha2_case.n=1}
    A_q(\bfn,\bfm,3)
    \leq q^{m+t-1}\frac{q^m-1+(t-1)(q-1)+(-\varepsilon-1)(q-\varepsilon-1)}{(q^m+(q-1)(t-1)+\varepsilon)(q^m+(q-1)(t-1)+\varepsilon-q)},
\end{equation}
where $\varepsilon=(t-1)\mod q$.

In the rest of this section, we analyze the 
bounds of Theorems~\ref{thm:induced} and~\ref{thm:non-induced}, comparing them
with the bound in (\ref{eq:alpha2_case.n=1}).
 
To evaluate the Sphere-Packing bound of Theorem~\ref{thm:non-induced} for $d=3$, we require the cardinality of a ball in $\spac$ of radius $1$, which is exactly $\delta+1$, where $\delta$ is the degree of the graph as in the proof of Theorem~\ref{thm:rt-bound_family}. This leads us to
$$A_q(\bfn,\bfm,3)\leq \frac{q^{m+t-1}}{q^m+(t-1)(q-1)}.$$

\begin{lemma}\label{lem:n=1.beat_SP3} Let $\bfn=(1,\dots,1)$ and $\bfm=(m,1,\dots,1)$.
Then the approximation for $A_q(\bfn,\bfm,3)$ given by the Ratio-Type bound~(\ref{eq:alpha2_case.n=1}) is not worse than the approximation using the Sphere-Packing bound of Theorem~\ref{thm:non-induced}.
\end{lemma}

\begin{proof} It suffices to prove that the inequality
\begin{equation}\label{eq:spectral_SP3}
    q^{m+t-1}\frac{q^m-1+(t-1)(q-1)+(-\varepsilon-1)(q-\varepsilon-1)}{(q^m+(q-1)(t-1)+\varepsilon)(q^m+(q-1)(t-1)+\varepsilon-q)}\leq \frac{q^{m+t-1}}{q^m+(t-1)(q-1)}
\end{equation}
always holds. By multiplying both sides of the inequality by $\frac{1}{q^{m+t-1}}$ and both denominators, all of which are all positive values, and subtracting $(q^m+(t-1)(q-1))^2$ from both sides, we obtain an inequality equivalent to~(\ref{eq:spectral_SP3}), namely
\begin{multline*}
    (-1+(-\varepsilon-1)(q-\varepsilon-1))(q^m+(t-1)(q-1)) \leq \\  (2\varepsilon-q)(q^m+(t-1)(q-1))+\varepsilon(\varepsilon-q). 
\end{multline*}
The latter is equivalent to
$\varepsilon(\varepsilon-q)(q^m+(q-1)(t-1)-1) \leq 0$, which is always true since $\varepsilon\geq 0$, $q^m+(q-1)(t-1)-1>0$, and $\varepsilon-q<0$.
\end{proof}

It is easy to see that Singleton bound 
of Theorem~\ref{thm:non-induced}, 
in the case we are considering in this section, reduces to
$$A_q(\bfn,\bfm,3)\leq q^{t-2}.$$

\begin{lemma}\label{lem:n=1.beat_S3} 
Let $\bfn=(1,\dots,1)$, $\bfm=(m,1,\dots,1)$,
and $t\geq q^m+2$.
Then the approximation for $A_q(\bfn,\bfm,3)$ given by the Ratio-Type bound~(\ref{eq:alpha2_case.n=1}) is not worse than the approximation using the Singleton bound of Theorem~\ref{thm:non-induced}.
\end{lemma}
\begin{proof}
This will follow as a corollary of Theorem~\ref{thm:msrd} below.
\end{proof}

Recall that the Total-Distance bound only applies when $$d>\sum\limits_{i=1}^t n_i -\sum\limits_{i=1}^t q^{-m_i}.$$ In the case analyzed in this section, this reduces to $3>t-q^{-m}-(t-1)/q$, which is in turn equivalent to $(t-1)q^{m-1}+1>(t-3)q^m$.

\begin{lemma}\label{lem:n=1.beat_TD3}
Let $\bfn=(1,\dots,1)$, $\bfm=(m,1,\dots,1)$, and $t\neq3$. Then the approximation for
$A_q(\bfn,\bfm,3)$
given by the Ratio-Type bound~(\ref{eq:alpha2_case.n=1}) is not worse than the approximation using the Total Distance bound of Theorem~\ref{thm:non-induced}, whenever the latter bound is applicable.
\end{lemma}

\begin{proof}
Observe that the inequality $(t-1)q^{m-1}+1>(t-3)q^m$ holds if and only if either $(t-1)q^{m-1}=(t-3)q^m$ or $(t-1)q^{m-1}>(t-3)q^m$. It is easy to see that, assuming $t\geq2$ and $q\geq2$, the latter inequality only holds for $t=3$, while the only integer solutions of the equation are $(t,q)=(4,3)$ and $(t,q)=(5,2)$.

To prove the lemma, we need to show that the inequality
\begin{equation}\label{eq:alpha2_TD3}
    q^{m+t-1}\frac{q^m-1+(t-1)(q-1)+(-\varepsilon-1)(q-\varepsilon-1)}{(q^m+(q-1)(t-1)+\varepsilon)(q^m+(q-1)(t-1)+\varepsilon-q)}\leq \frac{3}{3-t+\frac1{q^m}+\frac{t-1}q}
\end{equation}
holds for all cases when the Total Distance bound is applicable, excluding the case $t=3$.
If $t=4$ and $q=3$ then $\varepsilon=(4-1)\mod3=0$, and inequality~(\ref{eq:alpha2_TD3}) reduces to
\begin{align*}
  3^{m+3}\frac{3^m-1+6-2}{(3^m+6)(3^m+6-3)} &\leq \frac{3\cdot 3^m}{3\cdot3^m-4\cdot3^m+1+3\cdot3^{m-1}} \\
  3^2\frac{1}{3^m+6} &\leq 1,
\end{align*}
which is true for all positive $m$.
If $t=5$ and $q=2$, then $\varepsilon=(5-1)\mod2=0$. The inequality~(\ref{eq:alpha2_TD3}) becomes
\begin{align*}
    2^{m+4}\frac{2^m-1+4-1}{(2^m+4)(2^m+4-2)} &\leq \frac{3\cdot2^m}{3\cdot2^m-5\cdot2^m+1+4\cdot2^{m-1}}\\
    2^4\frac{1}{2^m+4} &\leq 3\\
    16-12&\leq 3\cdot2^m,
\end{align*}
which is true for all $m\geq1$.
\end{proof}

Note that if $t\geq q^m$, as in the assumptions of Lemma~\ref{lem:n=1.beat_S3}, then the condition $t\neq3$ of Lemma~\ref{lem:n=1.beat_TD3} follows automatically from $m>1$.

Using arguments similar to those introduced in Lemmas~\ref{lem:n=1.beat_SP3} and \ref{lem:n=1.beat_TD3}, one obtains that
for any $q$, $m$, and $t$ for which Lemma~\ref{lem:n=1.beat_S3} holds, our Ratio-Type bound is not worse than Induced Singleton and Induced Hamming bounds from Theorem~\ref{thm:induced} as well. It is also not worse than Induced Plotkin bound, whenever the latter is applicable. By combining all these observations we get the following results.

\begin{theorem}\label{thm:n=1.beat_everyone} 
Let $\bfn=(1,\dots,1)$, $\bfm=(m,1,\dots,1)$, and $t\geq q^m+2$. Then the approximation for
$A_q(\bfn,\bfm,3)$
given by the Ratio-Type bound~(\ref{eq:alpha2_case.n=1}) is not worse than the approximation using the Sphere-Packing, Singleton, Total-Distance, Induced Singleton, Induced Hamming, and Induced Plotkin bound for any $(t,q,m)$, whenever these bounds are applicable.
\end{theorem}


\subsection{The case \texorpdfstring{$\bfn=(2,1,\dots,1)$}{n=(2,1,...,1)}, \texorpdfstring{$\bfm=(m,1,\dots,1)$}{m=(m,1,...,1)}, and \texorpdfstring{$d=3$}{d=3}}\label{subsection:n=2}
In this subsection we consider $\spac$ with $\bfn=(2,1,\dots,1)$ and $\bfm=(m,1,\dots,1)$ for some $t\geq2$ and $m\geq 2$. According to Proposition~\ref{prop:prod}, the graph is the Cartesian product of the bilinear form graph $\Gamma(\Mat(2,m,\Fq))$ and the Hamming-metric graph $H(t-1,q)$ with $q^{t-1}$ vertices. 

From the bound in~(\ref{eq:alpha2_bound}) we can deduce the bound on the size of a code with minimum distance $d=3$:

{\small
\begin{equation}\label{eq:alpha2_case.n=2}
A_q(\bfn,\bfm,3)\leq q^{2m+t-1}\frac{q^m + q^{m+1} - t + q (t-3 - \varepsilon) + (1 + \varepsilon)^2}{(1 + q^m (q+1) + q (t-3) - t + \varepsilon) (1 + q^m (q+1) + q (t-2) - t + \varepsilon)}.
\end{equation}
}

\begin{lemma}\label{lem:n=2.beat_SP3} Let $\bfn=(2,1,\dots,1)$ and $\bfm=(m,1,\dots,1)$.
Then the approximation for $A_q(\bfn,\bfm,3)$ given by the Ratio-Type bound~(\ref{eq:alpha2_case.n=2}) is not worse than the approximation using Sphere-Packing bound of Theorem~\ref{thm:non-induced}.
\end{lemma}

\begin{proof}
    The Sphere-Packing bound reads
$$A_q(\bfn,\bfm,3)\leq \frac{q^{2m+t-1}}{(q+1)(q^m-1)+(t-1)(q-1)}.$$
Therefore the inequality we need to prove is:
{\footnotesize{
\begin{equation*}
    \frac{q^{2m+t-1}(q^m + q^{m+1} - t + q (t-3 - \varepsilon) + (1 + \varepsilon)^2)}{(1 + q^m (q+1) + q (t-3) - t + \varepsilon) (1 + q^m (q+1) + q (t-2) - t + \varepsilon)} \leq \frac{q^{2m+t-1}}{(q+1)(q^m-1)+(t-1)(q-1)}.
\end{equation*}
}}
This is equivalent to
{\footnotesize{
$$\frac{q^m + q^{m+1} - t + q (t-3 - \varepsilon) + (1 + \varepsilon)^2}{(1 + q^m (q+1) + q (t-3) - t + \varepsilon) (1 + q^m (q+1) + q (t-2) - t + \varepsilon)} \leq \frac1{(q+1)(q^m-1)+(t-1)(q-1)},$$}}
which is in turn equivalent to
$$t \leq \varepsilon(q^m(q+1)+q(q-1)(t-2)-3)(q-\varepsilon)+q^m(q+1)+q(t-3)+2\varepsilon+1.$$
The latter inequality is clearly true for all $t\geq2$ and $q\geq2$, since $t\leq q^m(q+1)+q(t-3)$ under these conditions.
\end{proof}

\begin{lemma}\label{lem:n=2.beat_S3}
Suppose $\bfn=(2,1,\dots,1)$, $\bfm=(m,1,\dots,1)$, and $t> (q^{2 m}-q^{m+1}-q^m+2 q-1)/(q-1)$.
Then the approximation for
$A_q(\bfn,\bfm,3)$
given by the Ratio-Type bound in~(\ref{eq:alpha2_case.n=2}) is not worse than the approximation given by the Singleton bound of Theorem~\ref{thm:non-induced}.
\end{lemma}
\begin{proof}
This will again follow from  Theorem~\ref{thm:msrd}.
\end{proof}

Similar to Section~\ref{subsection:n=1}, we can derive conditions under which the Ratio-Type bound~(\ref{eq:alpha2_case.n=2}) cannot perform worse than the Induced Singleton, Induced Hamming, Induced Plotkin, and Total-Distance bounds for sum-rank-metric codes described in Section~\ref{sec:code-bounds}.

\begin{theorem}\label{thm:n=2.beat_everyone} Let $d=3$, $\bfn=(2,1,\dots,1)$, $\bfm=(m,1,\dots,1)$, and $\smash{t> \frac{q^{2 m}-q^{m+1}-q^m+2 q-1}{q-1}}$.
Then the approximation given by 
the Ratio-Type bound in~(\ref{eq:alpha2_case.n=2}) is not worse than the approximation given by the Sphere-Packing, Singleton, Total-Distance, Induced Singleton, Induced Hamming, and Induced Plotkin bounds,
whenever these bounds 
are applicable.
\end{theorem}

We conclude this subsection with two remarks on other techniques that can be applied in the context investigated in this paper.

\subsection{Some Computational Results}\label{subsec:computationalboundcomparisson}

We include some tables to illustrate the performance of our eigenvalue bounds, in comparison with the best currently known bounds.
As the tables show, our bound improve on the state of the art for several parameters.

Table~\ref{tab:d=3} and Table~\ref{tab:d=4} list the graphs $\Gamma(\spac)$ on $|V|$ vertices, $|V|\leq1000$, which are not distance-regular and for which the Ratio-Type bound on $\alpha_2$ or $\alpha_3$ does not perform worse than the bounds of
Theorems~\ref{thm:induced} and~\ref{thm:non-induced}.

The columns labeled ``$\alpha_2$'' and ``$\alpha_3$'' contain the exact value of the $k$-independence number, for $k=2,3$.
We write ``time'' if the computation takes a long time on a standard laptop.
Only cases for which $\alpha_k\geq2$ are included.

The columns $\vartheta_k$, $k=2,3$, contain the Lov\'asz theta function of the graph that has the same vertex set as $\Gamma(\spac)$ and where two vertices are adjacent if and only if the distance between them in $\Gamma(\spac)$ is at most $k$. This value also provides an upper bound on $\alpha_k$; see Remark~\ref{rem:theta}.

The columns ``RT$_{d-1}$'' (for $d=3,4$) give the values of the Ratio-Type bound, given by Theorems~\ref{thm:hoffman-k=2} and~\ref{thm:hoffman-k=3}.
The columns ``iS$_d$'', ``iH$_d$'', ``iP$_d$'', ``iE$_d$'', ``S$_d$'', ``SP$_d$'', ``PSP$_d$'', ``TD$_d$'' (for $d=3,4$) give the values of the Induced Singleton, Induced Hamming, Induced Plotkin, Induced Elias, Singleton, Sphere-Packing, Projective Sphere-Packing, and Total-Distance bounds, respectively. These can be found in Section~\ref{sec:srmc}. 
When a bound is not applicable,
we return $0$ as value.
By definition, the values of ``SP$_3$'' and ``PSP$_3$'' are the same, so only ``SP$_3$'' is included in the table. The entries of the column ``RT$_{d-1}$'' in bold correspond to the instances in which the Ratio-Type bound performs better than the applicable code-theoretical bounds.

Since the code cardinality is always an integer, the columns only contain integer values, which are sometimes obtained by taking the floor of the real value given by the bound.

\begin{table}[!htbp]
{\tiny
    \centering
\[
\begin{array}{|ccll|c|c|cc|cccc|ccc|}
\hline
t & q & \bfn & \bfm & |V| & \text{RT}_2 & \alpha_2 & \vartheta_2 & \text{iS}_3 & \text{iH}_3 & \text{iP}_3 & \text{iE}_3 &  \text{S}_3 & \text{SP}_3 &  \text{TD}_3 \\
\hline
 2 & 2 &             (2, 1)            &             (2, 1)            &  32 &   2    &   2 &  2  &  4  &  6  &  4 & 6  &  2  &  2   &  2\\
 2 & 2 &             (2, 1)            &             (2, 2)            &  64 &   4    &   4  & 4  &  4  &  6  &  4  & 6 & 4  &  4   &  4  \\
 2 & 2 &             (2, 1)            &             (3, 3)            & 512 &   8    &   8  & 8 &  8  &  23 &  8 & 17 &  8  &  17  &  8 \\  
 2 & 2 &             (2, 2)            &             (2, 2)            & 256 & \bf{11}&  9   & 10  &  16 &  19 &  0 & 34  &  16 &  13  &  0\\  
 2 & 3 &             (2, 1)            &             (2, 2)            & 729 &   9    &   9  & 9  &  9  &  29 &  9  & 20 &  9  &  17 &  9  \\
 3 & 2 &           (1, 1, 1)           &           (2, 1, 1)           &  16 &   2    &   2  & 2  &  4  &  6  &  4  & 6 &  2  &  2 &  2  \\
 3 & 2 &           (2, 1, 1)           &           (2, 2, 2)           & 256 &   16   &   16  & 16 &  16 &  19 &  0  & 34 &  16 &  16 & 0 \\ 
 3 & 2 &           (2, 2, 1)           &           (2, 2, 1)           & 512 &  25  & 18  & 20 &  64 &  64 &  0 & 151 &  32 &  25 &  0 \\ 
 4 & 2 &          (1, 1, 1, 1)         &          (2, 1, 1, 1)         &  32 &   4    &   4   & 4 &  16 &  19 &  0  & 32 &  4  &  4  &  4  \\
 4 & 2 &          (2, 1, 1, 1)        &          (2, 2, 1, 1)         & 256 &   16   &   16 & 16  &  64 &  64 &  0  & 151 &  16 &  17 &  0 \\ 
 4 & 2 &          (2, 1, 1, 1)         &          (2, 2, 2, 1)         & 512 & \bf{28}& \text{time} & 24 & 64 &  64 &  0 & 151  &  32 &  30 & 0 \\  
 5 & 2 &        (1, 1, 1, 1, 1)        &        (2, 1, 1, 1, 1)      &  64 &   8    &   8  & 8  &  64 &  64 &  0  & 64 &  8  &  8  & 12 \\
 5 & 2 &        (2, 1, 1, 1, 1)        &        (2, 2, 1, 1, 1)        & 512 &   32   &   32 & 32 & 256 & 215 &  0  & 512 & 32 &  32 & 0  \\ 
 6 & 2 &     (1, 1, 1, 1, 1, 1) &  (2, 1, 1, 1, 1, 1)  & 128 & \bf{12}  &   9  & 12  & 128 & 128 &  0  & 128 &  16 &  14 &  0  \\
 6 & 2 &   (2, 1, 1, 1, 1, 1)    &    (2, 1, 1, 1, 1, 1)   & 512 &  32  &   \text{time}  & 32  & 512 & 512 &  0 & 512 &  32 &  34 & 0  \\
 7 & 2 &  (1, 1, 1, 1, 1, 1, 1)  &  (2, 1, 1, 1, 1, 1, 1) & 256 &  25  &  \text{time}  & 25 & 256 & 256 &  0  & 256 &  32 &  25 & 0 \\ 
 8 & 2 &    (1, 1, 1, 1, 1, 1, 1, 1)   &    (2, 1, 1, 1, 1, 1, 1, 1)   & 512 &  \bf{42} &  \text{time}  & 42  & 512 & 512 &  0  & 512 &  64 &  46 &  0  \\
\hline
\end{array}
\]
}
    \caption{The non-distance-regular sum-rank-metric graphs with $\leq1000$ vertices for which our eigenvalue bound for $A_q(\bfn,\bfm,3)$  is not worse than the approximation by a known bound. Improvements with respect to the best bound currently known are in bold.
    }
    \label{tab:d=3}
\end{table}

\begin{table}[!htbp]
\centering
{\tiny
\[
\begin{array}{|ccll|c|c|cc|cccc|cccc|}
\hline
 t & q & \bfn & \bfm& |V| & \text{RT}_3  & \alpha_3  & \vartheta_3 & \text{iS}_4 & \text{iH}_4 & \text{iP}_4 & \text{iE}_4 & \text{S}_4 & \text{SP}_4 & \text{PSP}_4 & \text{TD}_4 \\
\hline
 4 & 2 &          (1, 1, 1, 1)         &          (2, 1, 1, 1)         &  32 &   2    &   2  & 2   &  4  &  19 &  4 & 5 & 2  &  4  &  2   &  2 \\
 4 & 2 &          (2, 1, 1, 1)         &          (2, 1, 1, 1)         & 128 &   4    &   4  & 4  &  16 &  64 &  16 & 27 & 4  &  9  &  8   &  4 \\
 6 & 2 &   (2, 1, 1, 1, 1, 1)    &    (2, 1, 1, 1, 1, 1)   & 512  &  16  &   9   & 16  & 256 & 512 &  0 & 407 & 16 &  34 &  32  &  0 \\
\hline
\end{array}
\]
}
    \caption{The non-distance-regular sum-rank-metric graphs with $\leq1000$ vertices for which our eigenvalue bound for $A_q(\bfn,\bfm,4)$ is not worse than the approximation by a known bound. 
    Improvements with respect to the best bound currently known are in bold.
    }
    \label{tab:d=4}
\end{table}

\newpage
\section{Applications to MSRD Codes}\label{section:applications}
\label{sec:6}
We conclude this paper with a short section addressing the problem of the block length of a code meeting the 
Singleton bound in Theorem~\ref{thm:non-induced}. Such a code is called \textbf{MSRD} (\textbf{Maximum Sum-Rank-Distance Code}). MSRD codes generalize the celebrated MDS codes in the Hamming metric, which are codes $\mC \subseteq \F_q^n$ whose dimension $k$ and minimum distance $d$ satisfy $k =n-d+1$. 
It is a long-standing open problem to prove that
the length of a linear MDS code cannot exceed $q+1$, with very few exceptions.
This is a celebrated conjecture by Segre; see~\cite{segre1955curve}.
The block length of linear MSRD codes was investigated in~\cite{byrne2021fundamental}, with existence criteria obtained via weight distributions; see~\cite[Section~6]{byrne2021fundamental}. Later in~\cite{SZ2022} a stronger upper bound on the block length was introduced for certain parameters. In this section we apply eigenvalue methods and obtain results that are different from those of~\cite{byrne2021fundamental} in two ways: the parameters to which our eigenvalue bounds apply are different, and we are able to apply our results to codes that are not necessarily linear. In particular, linearity is not one of the assumptions of our results.

The next result shows that
certain MSRD codes with minimum distance $d=3$ have block length~$t$ relatively small. The parameters investigated in~\cite[Theorem VI.12]{byrne2021fundamental} are different from those studied here.

\begin{theorem}\label{thm:msrd}
    Let $\bfn=(n,1,\dots,1)$ and $\bfm=(m,1,\dots,1)$ for some $t$, and suppose there exists an MSRD code $\mC \subseteq \spac$ of minimum distance $d=3$. Then
    $$t\leq\begin{cases}
        1+q^m & \mbox{if } n=1, \\
        \frac{q^{2 m}-q^{m+1}-q^m+2 q-1}{q-1} & \mbox{if } n=2, \\
        \frac{q^{2 m+1}-q^{2 m}-q^{m+n}+q^m+q^n+q^2-3 q+1}{(q-1)^2} & \mbox{if } n>2.
    \end{cases}$$
\end{theorem}

\begin{proof}
    We first note that, by definition and  Theorem~\ref{thm:non-induced}, we have
    $$|\mC|=\begin{cases}
    q^{t-2} & \mbox{if } n=1,\\
    q^{t-1} & \mbox{if } n=2,\\
    q^{m(n-2)+t-1} & \mbox{if } n>2.
    \end{cases}$$
Here we provide the proof for the case $n=1$. The cases $n=2$ and $n>2$ can be shown in a similar way. Consider the inequality
$$q^{m+t-1}\frac{q^m-1+(q-1)(t-1)+(-\varepsilon-1)(q-\varepsilon-1)}{(q^m+(q-1)(t-1)+\varepsilon)(q^m+(q-1)(t-1)+\varepsilon-q)} < q^{t-2}.$$
The left hand is the value of the Ratio-Type bound~(\ref{eq:alpha2_case.n=1}) for $n=1$, which is an upper bound on the size of $\mC$, assuming $\varepsilon=(t-1)\mod q$. Since $|\mC|=q^{t-2}$, if for some $t$ the inequality is true, then an MSRD code cannot exist for this $t$.
We divide both sides of the inequality by $q^{t-2}$ and rewrite it as follows:
{\small{
\begin{multline*}
-(q-1)^2t^2+(q-1) (q^{m+1} - 2 q^m + 3 q-2\varepsilon - 2)t+\\
+ (q-1) (q^m+1) (q^m - 2 q + 1) + \varepsilon^2 (q^{m+1}-1) - \varepsilon (q^{m+2}-2q^{m+1} + 2 q^m-3q + 2)<0.
\end{multline*}
}}
We replace $\varepsilon$ by a variable $\varepsilon_0$ that takes integer values in range from $0$ to $q-1$ and introduce the function
\begin{multline*}
    F(t,q,m,\varepsilon_0) = -(q-1)^2t^2+(q-1) (q^{m+1} - 2 q^m + 3 q-2\varepsilon_0 - 2)t+ \\ + (q-1) (q^m+1) (q^m - 2 q + 1)+ \varepsilon_0^2 (q^{m+1}-1) - \varepsilon_0 (q^{m+2}-2q^{m+1} + 2 q^m-3q + 2)
\end{multline*}
Note that $F(t,q,m,\varepsilon_0)$ is quadratic in $t$ with $F(t,q,m,\varepsilon_0)\to-\infty$ as $t$ approaches infinity. Then $F(t,q,m,\varepsilon_0)<0$ is true for
$$t>\frac{q \sqrt{4 \varepsilon_0 (\varepsilon_0+1) q^{m-1}-2 (2 \varepsilon_0+1) q^m+q^{2 m}+1}-2 \varepsilon_0+q^{m+1}-2 q^m+3 q-2}{2 (q-1)}.$$
Finally, we remove the dependency on $\varepsilon_0$ by assigning it a value in the range from $0$ to $q-1$ that maximizes the right hand side. For this, note that the expression under the square root can be seen as a quadratic function of $\varepsilon_0$ with its minimum at $\varepsilon_0=\frac{q-1}2$. Hence this expression is maximized by assigning either $0$ or $q-1$ to  $\varepsilon_0$. Since $\varepsilon_0=0$ also maximizes the term $-2\varepsilon_0$, this is the assignment we select. By substituting $\varepsilon_0=0$ we get
$$t>q^m+1.$$
Hence an MSRD code can only exist if $t\leq q^m+1$.
\end{proof}

Clearly, MSRD codes cannot exist when the Singleton bound of Theorem~\ref{thm:non-induced}
is not tight.
Table~\ref{tab:no_msrd} contains a list of parameters for sum-rank-metric graphs $\Gamma(\spac)$ on $|V|$ vertices, $|V|\leq 10^5$, which cannot contain an MSRD code because
of the constraints imposed by our Ratio-Type bounds from Theorems~\ref{thm:hoffman-k=2} and \ref{thm:hoffman-k=3}. We exclude graphs
with $\bfn=(1,\dots,1)$, i.e., situations that can be reduced to the Hamming metric with a mixed alphabet. Other graphs that are excluded from the table correspond to $\bfm=(m,\dots,m)$ for some integer $m$, 
as we could not find instances where our Ratio-Type bounds give a stronger existence criterion than that presented in~\cite{byrne2021fundamental}.

Table~\ref{tab:no_msrd} includes the values of the Ratio-Type bound ``RT$_{d-1}$'' as well as other bounds of Theorems~\ref{thm:induced} and~\ref{thm:non-induced}: Induced Singleton (``iS$_d$''), Induced Hamming (``iH$_d$''), Induced Elias (``iE$_d$''), Singleton (``S$_d$'', which would also be the size of an MSRD code in case it exists), Sphere-Packing (``SP$_d$''), and Projective Sphere-Packing bounds (``PSP$_d$''). Note that in each case presented in Table~\ref{tab:no_msrd}, both the Induced Plotkin bound of Theorem~\ref{thm:induced} and the Total Distance bound of Theorem~\ref{thm:non-induced} cannot be applied, since the conditions mentioned in the respective theorems are not satisfied. Finally, Table~\ref{tab:no_msrd} contains the parameters for which the value of the Ratio-Type bound is the \textit{only} bound that
outperforms the Singleton bound. In other words,
the table only contains parameters for which the  existence of an MSRD code cannot be excluded
by applying any of the known bounds.

\begin{table}[!htbp]
\tiny
    \centering
    \[
    \begin{array}{|ccll|ccc|ccc|ccc|}
\hline
 t  & q  & \bfn  & \bfm  & d  & |V| & \text{RT}_{d-1}  & \text{iS}_d  & \text{iH}_d  & \text{iE}_d  & \text{S}_d  & \text{SP}_d  & \text{PSP}_d  \\
\hline
 3  & 2  & (3, 2, 1)  & (3, 2, 2)  & 3  & 32768  & 494  & 4096  & 6096  & 15362  & 512  & 528  & 528  \\
 4  & 2  & (3, 1, 1, 1)  & (3, 3, 2, 2)  & 3  & 65536  & 989  & 4096  & 6096  & 15362  & 1024  & 1040  & 1040  \\
 4  & 2  & (2, 2, 2, 1)  & (2, 2, 2, 1)  & 4  & 8192  & 98  & 256  & 744  & 407  & 128  & 282  & 204  \\
 5  & 2  & (2, 2, 1, 1, 1)  & (2, 2, 2, 2, 1)  & 4  & 8192  & 107  & 256  & 744  & 407  & 128  & 315  & 292 \\
 5  & 2  & (2, 2, 2, 1, 1)  & (2, 2, 2, 1, 1)  & 4  & 16384  & 193  & 1024  & 2621  & 1419  & 256  & 546  & 409 \\
 5  & 2  & (2, 2, 2, 1, 1)  & (2, 2, 2, 2, 1)  & 4  & 32768  & 338  & 1024  & 2621  & 1419  & 512  & 1024  & 819 \\
 6  & 2  & (2, 1, 1, 1, 1, 1)  & (2, 2, 2, 2, 2, 1)  & 4  & 8192  & 119  & 256  & 744  & 407  & 128  & 356  & 512  \\
 6  & 2  & (2, 2, 1, 1, 1, 1)  & (2, 2, 2, 1, 1, 1)  & 4  & 8192  & 123  & 1024  & 2621  & 1419  & 128  & 327  & 292  \\
 6  & 2  & (2, 2, 1, 1, 1, 1)  & (2, 2, 2, 2, 1, 1)  & 4  & 16384  & 212  & 1024  & 2621  & 1419  & 256  & 606  & 585  \\
 6  & 2  & (2, 2, 1, 1, 1, 1)  & (2, 2, 2, 2, 2, 1)  & 4  & 32768  & 371  & 1024  & 2621  & 1419  & 512  & 1129  & 1170  \\
 6  & 2  & (2, 2, 2, 1, 1, 1)  & (2, 2, 2, 1, 1, 1)  & 4  & 32768  & 378  & 4096  & 9362  & 5026  & 512  & 1057  & 819  \\
 6  & 2  & (2, 2, 2, 1, 1, 1)  & (2, 2, 2, 2, 1, 1)  & 4  & 65536  & 673  & 4096  & 9362  & 5026  & 1024  & 1985  & 1638  \\
 7  & 2  & (2, 1, \dots, 1)  & (2, 1, \dots, 1)  & 4  & 1024  & 30  & 1024  & 1024  & 1024  & 32  & 64  & 64  \\
 7  & 2  & (2, 1, \dots, 1)  & (2, \dots, 2, 1, 1)  & 4  & 16384  & 235  & 1024  & 2621  & 1419  & 256  & 682  & 1024  \\
 7  & 2  & (2, 1, \dots, 1)  & (2,\dots, 2, 1)  & 4  & 32768  & 397  & 1024  & 2621  & 1419  & 512  & 1260  & 2048  \\
 7  & 2  & (2, 2, 1, \dots, 1)  & (2, 2, 2, 1, \dots, 1)  & 4  & 16384  & 246  & 4096  & 9362  & 5026  & 256  & 630  & 585  \\
 7  & 2  & (2, 2, 1, \dots, 1)  & (2, \dots, 2, 1, 1, 1)  & 4  & 32768  & 422  & 4096  & 9362  & 5026  & 512  & 1170  & 1170 \\
 7  & 2  & (2, 2, 1, \dots, 1)  & (2, \dots, 2, 1, 1)  & 4  & 65536  & 733  & 4096  & 9362  & 5026  & 1024  & 2184  & 2340  \\
 7  & 2  & (2, 2, 2, 1, \dots, 1)  & (2, 2, 2, 1, \dots, 1)  & 4  & 65536  & 755  & 16384  & 33825  & 18037  & 1024  & 2048  & 1638  \\
 8  & 2  & (2, 1, \dots, 1)  & (2, 1, \dots, 1)  & 4  & 2048  & 57  & 2048  & 2048  & 2048  & 64  & 120  & 128  \\
 8  & 2  & (2, 1, \dots, 1)  & (2,\dots, 2, 1, 1, 1)  & 4  & 32768  & 459  & 4096  & 9362  & 5026  & 512  & 1310  & 2048  \\
 8  & 2  & (2, 1, \dots, 1)  & (2, \dots, 2, 1, 1)  & 4  & 65536  & 799  & 4096  & 9362  & 5026  & 1024  & 2427  & 4096  \\
 8  & 2  & (2, 2, 1, \dots, 1)  & (2, 2, 2, 1, \dots, 1)  & 4  & 32768  & 467  & 16384  & 32768  & 18037  & 512  & 1213  & 1170  \\
 8  & 2  & (2, 2, 1, \dots, 1)  & (2, 2, 2, 2, 1, \dots, 1)  & 4  & 65536  & 819  & 16384  & 33825  & 18037  & 1024  & 2259  & 2340  \\
 9  & 2  & (2, 1, \dots, 1)  & (2, 1, \dots, 1)  & 4  & 4096  & 107  & 4096  & 4096  & 4096  & 128  & 227  & 256  \\
 9  & 2  & (2, 1, \dots, 1)  & (2, \dots, 2, 1, 1, 1, 1)  & 4  & 65536  & 924  & 16384  & 33825  & 18037  & 1024  & 2520  & 4096  \\
 9  & 2  & (2, 2, 1, \dots, 1)  & (2, 2, 2, 1, \dots, 1)  & 4  & 65536  & 936  & 65536  & 65536  & 65412  & 1024  & 2340  & 2340  \\
 10  & 2  & (2, 1, \dots, 1)  & (2, 1, \dots, 1)  & 4  & 8192  & 204  & 8192  & 8192  & 8192  & 256  & 431  & 512  \\
 10  & 2  & (2, 1, \dots, 1)  & (2, 2, 2, 2, 1, \dots, 1)  & 4  & 65536  & 1013  & 65536  & 65536  & 65412  & 1024  & 2621  & 4096  \\
 10  & 2  & (2, 2, 1, \dots, 1)  & (2, 2, 1, \dots, 1)  & 4  & 65536  & 1008  & 65536  & 65536  & 65536  & 1024  & 2427  & 2340  \\
 11  & 2  & (2, 1, \dots, 1)  & (2, 1, \dots, 1)  & 4  & 16384  & 384  & 16384  & 16384  & 16384  & 512  & 819  & 1024  \\
 12  & 2  & (2, 1, \dots, 1)  & (2, 1, \dots, 1)  & 4  & 32768  & 738  & 32768  & 32768  & 32768  & 1024  & 1560  & 2048  \\
 13  & 2  & (2, 1, \dots, 1)  & (2, 1, \dots, 1)  & 4  & 65536  & 1390  & 65536  & 65536  & 65536  & 2048  & 2978  & 4096  \\
\hline
    \end{array}
    \]
    \caption{Sum-rank-metric graphs on up to $10^5$ vertices that cannot contain an MSRD code, as shown by computing our Ratio-Type bounds. Existence cannot be excluded
    by applying the known bounds
    from Theorems~\ref{thm:induced} and~\ref{thm:non-induced}.}
    \label{tab:no_msrd}
\end{table}

\medskip

\subsection*{Acknowledgements}
We thank Alexander Gavrilyuk for his careful reading of this manuscript and suggestions.
Aida Abiad is partially supported by the Dutch Research Council via grants VI.Vidi.213.085, OCENW.KLEIN.475, and by the Research Foundation -- Flanders via grant 1285921N.
Antonina P. Khramova is supported by the Dutch Research Council via grant OCENW.KLEIN.475. Alberto Ravagnani is supported by the Dutch Research Council via grants VI.Vidi.203.045, OCENW.KLEIN.539,
and by the Royal Academy of Arts and Sciences of the Netherlands.

\bigskip

\bibliographystyle{abbrv}
\bibliography{references}

\end{document}